\documentclass[oribibl]{llncs}

\usepackage{math-standard}
\usepackage{ut-probability}
\usepackage{bbm}
\usepackage{tikz}

\newcommand{\bigOp}[1]{O_\pr \left(#1\right)}

\newcommand\numberthis{\addtocounter{equation}{1}\tag{\theequation}}

\begin{document}

\title{Upper bounds for number of removed edges in the Erased Configuration Model}
\author{Pim van der Hoorn\inst{1} \and Nelly Litvak\inst{1}}
\institute{
	University of Twente, Department of Electrical Engineering, Mathematics and Computer Science \\
	\email{w.l.f.vanderhoorn@utwente.nl}
}

\maketitle

\begin{abstract}
	Models for generating simple graphs are important in the study of real-world complex networks. A
	well established example of such a model is the erased configuration model, where each node 
	receives a number of half-edges that are connected to half-edges of other nodes at random, and 
	then self-loops are removed and multiple edges are concatenated to make the graph simple. 
	Although asymptotic results	for many properties of this model, such as the limiting degree 
	distribution, are known, the exact speed of convergence in terms of the graph sizes remains 
	an open question. We provide a first answer	by analyzing the size dependence of the average 
	number of removed edges in the erased configuration model. By combining known upper bounds with a
	Tauberian Theorem we obtain upper bounds for the number of removed edges, in terms of the size of
	the graph. Remarkably, when the degree distribution follows a power-law, we observe three scaling 
	regimes, depending on the power law exponent. Our results provide a strong theoretical basis for 
	evaluating finite-size effects in networks.
\end{abstract}

\section{Introduction}

The use of complex networks to model large systems has proven to be a powerful tool in recent 
years. Mathematical and empirical analysis of structural properties of such networks, such as graph
distances, centralities, and degree-degree correlations, have received vast attention in recent 
literature. A common approach for understanding these properties on real-world networks, is to 
compare them to those of other networks which have the same basic characteristics as the network 
under consideration, for instance the distribution of the degrees. Such test networks are usually 
created using random graph models. An important property of real-world networks is that they are 
usually simple, i.e. there is at most one edge between any two nodes and there are no self-loops. 
Hence, random graph models that produce simple graphs are of primary interest from the application
point of view.

One well established model for generating a graph with given degree distribution is the 
configuration model~\cite{Bollobas1980,Molloy1995,Newman2001}, which has been studied extensively in 
the literature~\cite{Britton2006,VanDerHofstad2007,Hoorn2014,Hoorn2015}. In this model, 
each node first receives a certain number of half-edges, or stubs, and then the stubs are connected
to each other at random. Obviously, multiple edges and self-loops may appear during the random 
wiring process. It is well-known that when the degree distribution has finite variance, the graph will be 
simple with positive probability, so a simple graph can be obtained by repeatedly applying the 
model until the resulting graph is simple. However, when the variance of the degrees is infinite 
the resulting graph will, with high probability, not be simple. To remedy this, one can remove 
self-loops and concatenate the multiple edges to make the graph simple. This version is know as the
erased configuration model. Although removal of edges impacts the degree distribution, it has been 
shown that asymptotically the degree distribution is unchanged. For a thorough systematic treatment 
of these results we refer the reader to~\cite{VanDerHofstad2007}.

An important feature of the configuration model is that, conditioned on the graph being simple, it 
samples a graph uniformly from among all simple graphs with the specified degree distribution. 
This, in combination with the neutral wiring in the configuration model, makes it a crucial model for 
studying the effects of degree distributions on the structural properties of the networks, such as,
for instance, graph distances~\cite{Esker2005,Hofstad2005,Hofstad2005a,Molloy1998} and epidemic 
spread~\cite{Andersson1998,Ferreira2012,Lee2013}. 

We note that there are several different methods for generating simple graphs, sampled uniformly
from the set of all simple graphs with a given degree sequence. A large class of such models use
Markov-Chain Monte Carlo methods for sampling graphs uniformly from among all graphs with a 
given set of constraints, such as the degree sequence. These algorithms use so-called edge swap or 
switching steps,\cite{Artzy-Randrup2005,Maslov2002,Tabourier2011}, each time a pair of edges is 
sampled and swapped, if this is allowed. The main problem with this method are the limited 
theoretical results on the mixing times, in~\cite{Cooper2007} mixing times are analyzed, but only 
for regular graphs. Other methods are, for instance, the sequential algorithms proposed in 
\cite{Blitzstein2011sequential,DelGenio2010} which have complexity $O(E N^2)$ and $O(E N)$, 
respectively, where $N$ is the size of the graph and $E$ denotes the number of edges. The erased
configuration model however,is well studied, with strong theoretical results and is easy to 
implement.

In a recent study \cite{Schlauch2015}, authors compare several methods, including the above 
mentioned Markov-Chain Monte Carlo methods, for creating test graphs for the analysis of
structural properties of networks. The authors found that the number of removed edges did not 
impact the degree sequence in any significant way. However, several other measures on the graph, 
for instance average neighbor degree, did seem to be altered by the removal of self-loops and 
double-edges. This emphasizes the fact that asymptotic results alone are not sufficient. The 
analysis of networks requires a more detailed understanding of finite-size effects in simple random
graphs. In particular, it is important to obtain a more precise characterization for 
dependence of the number of erased edges on the graph size, and their impact on other characteristics of the graph.

In our recent work~\cite{Hoorn2015} we analyzed the average number of removed edges in order to 
evaluate the degree-degree correlations in the directed version of the erased configuration model. 
We used insights obtained from several limit theorems to derive the scaling in terms of the graph 
size. Here we make these rigorous by proving three upper bounds for the average number of removed
edges in the undirected erased configuration model with regularly varying degree distribution. We 
start in Section 2 with the formal description of the model. Our main result is stated in Section 3
and the proofs are provided in Section 4.

\section{Erased Configuration Model}

The Erased Configuration Model (ECM) is an alteration of the Configuration Model (CM), which
is a random graph model for generating graphs of size $n$ with either prescribed degree sequence or 
degree distribution. Given a degree sequence ${\bf D}_n$ such that $\sum_{i = 1}^n D_i$ is 
even, the degrees of each node are represented as stubs and the graph is constructed by randomly 
pairing stubs to form edges. This will create a graph with the given degree sequence. 

In another version of the model, degrees are sampled independently from a given distribution, an 
additional stub is added to the last node if the sum of degrees is odd, and the stubs are connected
as in the case with given degrees. The empirical degree distribution of the resulting graph will 
then converge to the distribution from which the degrees were sampled as the graph size goes to 
infinity, see for instance~\cite{VanDerHofstad2007}. 

When the degree distribution has finite variance, the probability of creating a simple graph with 
the CM is bounded away from zero. Hence, by repeating the model, one will obtain a simple graph 
after a finite number of attempts. This construction is called the Repeated Configuration Model 
(RCM). It has been shown that the RCM samples graphs uniformly from among all simple graphs with 
the given degree distribution, see Proposition 7.13 in \cite{VanDerHofstad2007}. 

When the degrees have infinite variance the probability of generating a simple graph with the CM
converges to zero as the graph size increases. In this case the ECM can be used, where after all 
stubs are paired, multiple edges are merged and self-loops are removed. This model is 
computationally far less expensive than the RCM since the pairing only needs to be done once 
while in the other case the number of attempts increases as the variance of the degree distribution 
grows. The trade-off is that the ECM removes edges, altering the degree sequence and hence the 
empirical degree distribution. Nevertheless it was proven, see \cite{VanDerHofstad2007}, that the 
empirical degree distribution for the ECM still converges to the original one as $n \to \infty$.

For our analysis we shall consider graphs of size $n$ generated by the ECM, where the degrees are 
sampled at random from a regularly varying distribution. We recall that $X$ is said to have a 
regularly varying distribution with finite mean if
\begin{equation}
	\Prob{X > k} = L(k) k^{-\gamma} \quad \text{with } \gamma > 1,
\label{eq:degree_distribution}
\end{equation}
where $L$ is a slowly varying function, i.e. $\lim_{x \to \infty} L(tx)/L(x) = 1$ for all $t > 0$. 
The parameter $\gamma$ is called the exponent of the distribution.  

For $n \in \N$ we consider the degree sequence ${\bf D}_n$ as a sequence of i.i.d. samples from 
distribution~\eqref{eq:degree_distribution}, let $\mu = \Exp{D}$ and denote by $L_n = 
\sum_{i = 1}^n D_i$ the sum of the degrees. Formally we need $L_n$ to be even in order to have a 
graphical sequence, in which case $L_n/2$ is the number of edges. This can be achieved by 
increasing the degree of the last node $D_n$ by one if the sum is odd. This alteration adds a term 
uniformly bounded by one which does not influence the analysis. Therefore we can omit this and 
treat the degree sequence ${\bf D}_n$ as an i.i.d. sequence. 

For the analysis we denote by $E_{ij}$ the number of edges between two nodes, $1 \le i,j \le n$, 
created by the CM and by $E_{ij}^c$ the number of edges between the two nodes that where removed by
the ECM. Furthermore,  we let $\mathbb{P}_n$ and $\mathbb{E}_n$ be, respectively, the probability 
and expectation conditioned on the degree sequence ${\bf D}_n$.

\section{Main result}

The main result of this paper is concerned with the scaling of the average number of erased edges
in the ECM. It was proven in~\cite{Hoorn2014} that 
\begin{equation}\label{eq:erased_edges_convergence}
	\frac{1}{L_n} \sum_{i,j} \Expn{E_{ij}^c} \plim 0 \quad \text{as } n \to \infty,
\end{equation} 
where $\plim$ denotes convergence in probability. This result states that the average number of 
removed edges converges to zero as the graph size grows, which is in agreement with the convergence 
in probability of the empirical degree distribution to the original one. However, until now there
have not been any results on the speed of convergence in~\eqref{eq:erased_edges_convergence}. In 
this section we will state our result, which establishes upper bound on the scaling of the average
number of erased edges.

To make our statement rigorous we first need to define what we mean by scaling for a 
random variable.

\begin{definition}\label{def:probabilistic_bigO}
	Let $(X_n)_{n \in \N}$ be sequences of random variables and let 
	$\rho \in \R$. Then we define
	\[
		X_n = \bigOp{n^\rho} \iff \text{ for all } \varepsilon > 0 \quad 
		n^{-\rho - \varepsilon} X_n \plim 0, \quad \text{as } n \to \infty.
	\]
\end{definition}

We are now ready to state the main result on the scaling of the average number of erased edges
in the ECM

\begin{theorem}\label{thm:main_result}
	Let $G_n$ be a graph generated by the ECM with degree distribution
	\eqref{eq:degree_distribution}, let $L_n$ be the sum of the degrees and denote by $E_{ij}^c$ the 
	number of removed edges from $i$ to $j$. Then
	\begin{equation}
		\frac{1}{L_n} \sum_{i,j = 1}^n \Expn{E_{ij}^c} = 
		\begin{cases}
			\bigOp{n^{\frac{1}{\gamma} - 1}} &\mbox{ if } 1 < \gamma \le \frac{3}{2}, \\
			\bigOp{n^{\frac{4}{\gamma} - 3}} &\mbox{ if } \frac{3}{2} < \gamma \le 2,\\
			\bigOp{n^{-1}} &\mbox{ if } \gamma > 2.
		\end{cases}
	\label{eq:upper_bounds_erased_edges}
	\end{equation}
\end{theorem}

The proof of Theorem~\ref{thm:main_result} is given in the next section. The strategy of the proof 
is to establish two upper bounds for $\sum_{i,j = 1}^n \Expn{E_{ij}^c}/L_n$ for the case $1 < 
\gamma \le 2$, each of which scales as one of the first two terms from
\eqref{eq:upper_bounds_erased_edges}. Then it remains to observe that the term $n^{1/\gamma - 1}$ 
dominates $n^{4/\gamma - 3}$ when $1 < \gamma \le 3/2$ while the latter one dominates when 
$3/2 < \gamma < 2$. In addition, we prove the $n^{-1}$ scaling for $\gamma > 2$.

Theorem~\ref{thm:main_result} gives several insights into the behavior of the ECM.  First, consider
the case when the degrees have finite variance ($\gamma > 2$). Equation
\eqref{eq:upper_bounds_erased_edges} tells us that in that case the ECM will erase only a finite, 
in $n$, number of edges. For large $n$, this alters the degree sequence in a negligible way. We 
then gain the advantage that we need to perform the random wiring only once. In contrast, the RCM 
requires multiple attempts before a simple graph is produced. This will be a problem, especially as
$\gamma$ approaches $2$. 

An even more interesting phenomenon established by Theorem~\ref{thm:main_result} is the remarkable 
change in the scaling at $\gamma = 3/2$. Figure~\ref{fig:erased_edges_scaling} shows the exponent 
in the scaling term in \eqref{eq:upper_bounds_erased_edges} as a function of $\gamma$.
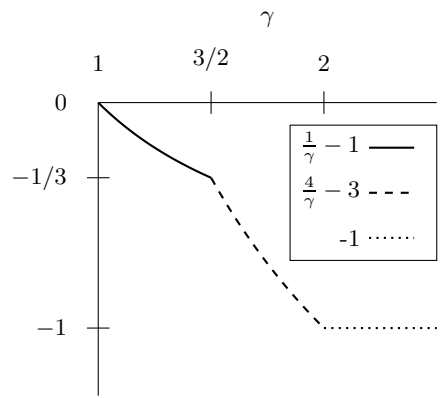
\begin{figure}%
	\centering
	\begin{tikzpicture}[scale=3]
	
	\draw (1,0) -- (2.5,0);
	\draw (1,0) -- (1,-1.3);
	
	\draw (1,0.1) node[above] {$1$};
	\draw (1.5,0.1) node[above] {$3/2$};
	\draw (1.5,-0.05) -- (1.5,0.05);
	\draw (2,0.1) node[above] {$2$};
	\draw (2,-0.05) -- (2,0.05);
	
	\draw (0.9,0) node[left] {$0$};
	\draw (0.9,-0.3333) node[left] {$-1/3$};
	\draw (0.95,-0.3333) -- (1.05,-0.3333);
	\draw (0.9,-1) node[left] {$-1$};
	\draw (0.95,-1) -- (1.05,-1);
	
	\draw (1.75,0.3) node[above] {$\gamma$};
	
	\draw[black,thick,domain=1:1.5] plot (\x, {(1/\x) - 1});
	\draw[black,dashed,thick,domain=1.5:2] plot (\x, {(4/\x) - 3});
	\draw[black,dotted,thick] (2,-1) -- (2.5,-1);
	
	\draw (1.85,-0.1) -- (2.5,-0.1) -- (2.5,-0.7) -- (1.85,-0.7) -- (1.85,-0.1);
	\draw (2.2,-0.2) node[left] {$\frac{1}{\gamma} - 1$};
	\draw[black,thick] (2.2,-0.2) -- (2.4,-0.2);
	\draw (2.2,-0.4) node[left] {$\frac{4}{\gamma} - 3$};
	\draw[black,dashed,thick] (2.2,-0.4) -- (2.4,-0.4);
	\draw (2.2,-0.6) node[left] {-1};
	\draw[black,dotted,thick] (2.2,-0.6) -- (2.4,-0.6);
	
\end{tikzpicture}
\caption{The scaling exponent of the average number of erased edges, as a function of $\gamma$.}%
\label{fig:erased_edges_scaling}%
\end{figure}
Notice that for small values of $\gamma$, the fraction of erased edges decreases quite slowly with 
$n$. For example, when $\gamma=1.1$ and $n=10^6$ then $n^{1/\gamma} \approx 284803$. Hence, a 
significant fraction of edges will be removed, so we can expect notable finite size effects even in
large networks. However, when $\gamma\ge 1.5$ the finite size effects are already very small and 
decrease more rapidly with $\gamma$.

It will be seen from the proofs in the next section that the upper bounds  for $\gamma>3/2$ in 
Theorem~\ref{thm:main_result} follow readily from the literature. Our main contribution is in the 
upper bound for $1 < \gamma<3/2$, which corresponds to many real-world networks. The proof uses a 
Central Limit Theorem and a Tauberian Theorem for regularly varying random variables. Note that 
when $1 < \gamma < 3/2$ the upper bound $n^{4/\gamma-3}$ is not at all tight and even increasing 
in $n$ for $\gamma<4/3$.

\section{Upper bounds for erased edges}\label{sec:upper_bounds}

Throughout this section we will use the Central Limit Theorem for regularly varying random 
variables also called the Stable Law CLT, see \cite{Whitt2002} Theorem 4.5.1. We summarize it 
below, letting $\dlim$ denote convergence in distribution, in the setting of non-negative 
regularly varying random variables.

\begin{theorem}[Stable Law CLT~\cite{Whitt2002}]\label{thm:stable_clt}
	Let $\{X_i: i \ge 1\}$ be an i.i.d. sequence of non-negative random variables with 
	distribution~\eqref{eq:degree_distribution} and $0 < \gamma < 2$. Then there exists a slowly 
	varying function $L_0$,	different from $L$, such that
	\[
		\frac{\sum_{i = 1}^n D_i - m_n}{L_0(n)n^{\frac{1}{\gamma}}} 
		\dlim \mathcal{S}_{\gamma},
	\]
	where $S_{\gamma}$ is a stable random variable and
	\begin{equation*}
		m_n = \begin{cases}
			0 &\mbox{if } 0 < \gamma < 1 \\
			n^{2}\Exp{\sin\left(\frac{X}{L_0(n) n}\right)} 
				&\mbox{if } \gamma = 1 \\
			n\Exp{X} &\mbox{if } 1 < \gamma < 2.
		\end{cases}
	\end{equation*}
\end{theorem}

From Theorem~\ref{thm:stable_clt} we can infer several scaling results using the following 
observation: By Slutsky's Theorem it follows that
\[
	\frac{\sum_{i = 1}^n X_i - m_n}{L_0(n) n^{\frac{1}{\gamma}}} \dlim S_{\gamma} \quad
	\text{as } n \to \infty
\]
implies that for any $\varepsilon > 0$,
\[
	n^{-\varepsilon}\frac{\sum_{i = 1}^n X_i - m_n}{L_0(n) n^{\frac{1}{\gamma}}} \plim 0
	\quad \text{as } n \to \infty.
\]
Hence $\left|\sum_{i = 1}^n X_i - m_n\right| = \bigOp{L_0(n) n^{1/\gamma}}$ and therefore, by 
Potter's Theorem, it follows that $\left|\sum_{i = 1}^n X_i - m_n\right| = \bigOp{n^{1/\gamma}}$. 
Finally, we remark that if $D$ has distribution~\eqref{eq:degree_distribution} with $1 < \gamma 
\le 2$, then $D^2$ has distribution~\eqref{eq:degree_distribution} with exponent  $1/2 < \gamma/2 
\le 1$. Summarizing, we have the following. 

\begin{corollary}\label{cor:clt_edges}
Let $G_n$ be a graph generated by the ECM with degree distribution
\eqref{eq:degree_distribution} and $1 < \gamma \le 2$, then
\begin{equation*}
	L_n = \bigOp{n}, \quad
	\left|\sum_{i = 1}^n D_i - \mu n\right| = \bigOp{n^{\frac{1}{\gamma}}} \, \text{ and } \,
	\sum_{i = 1}^n D_i^2 = \bigOp{n^{\frac{2}{\gamma}}}
\end{equation*}
\end{corollary}
The third equation also holds for $\gamma = 2$ since
\begin{align*}
	\sum_{i = 1}^n D_i^2 &= \left(\sum_{i = 1}^n D_i^2 - L_0(n) n^2 \Exp{\sin\left(
	\frac{D}{n L_0(n)}\right)}\right) + L_0(n) n^2 \Exp{\sin\left(	\frac{D}{n^1 L_0(n)}\right)} \\
	&\le \left(\sum_{i = 1}^n D_i^2 - n^2 L_0(n) \Exp{\sin\left(
	\frac{D}{n^1 L_0(n)}\right)}\right) + n \mu \\
	&= \bigOp{L_0(n) n} + n \mu = \bigOp{n}. 
\end{align*}

\subsection{The upper bounds $\bigOp{n^{\frac{4}{\gamma} - 3}}$ and $\bigOp{n^{-1}}$}

For the proof of the upper bounds we will use the following proposition.

\begin{proposition}[Proposition 7.10~\cite{VanDerHofstad2007}]\label{prop:first_upper_bound}
	Let $G_n$ be a graph generated by the CM and denote by $S_n$ and $M_n$, 
	respectively, the	number of self-loops and multiple edges. Then
	\begin{equation*}
		\Expn{S_n} \le \sum_{i = 1}^n \frac{D_i^2}{L_n} \quad \text{and} \quad
		\Expn{M_n} \le 2\left(\sum_{i = 1}^n \frac{D_i^2}{L_n}\right)^2.
	\end{equation*}
\end{proposition}

\begin{lemma}\label{lem:first_upper_bound}
	Let $G_n$ be a graph generated by the ECM with degree distribution
	\eqref{eq:degree_distribution}, then
	\begin{equation}
		\frac{1}{L_n} \sum_{i,j = 1}^n \Expn{E_{ij}^c} = 
		\begin{cases}
			\bigOp{n^{\frac{4}{\gamma} - 3}} &\mbox{ if } 1 < \gamma \le 2,\\
			\bigOp{n^{-1}} &\mbox{ if } \gamma > 2.
		\end{cases}
	\label{eq:erased_edges_bound_1}
	\end{equation}
\end{lemma}

\begin{proof}
We start by observing that 
\[
	\sum_{i,j = 1}^n E^c_{ij} = S_n + M_n,
\]
and hence it follows from Proposition~\ref{prop:first_upper_bound} that
\[
	\sum_{i,j = 1}^n \Expn{E^c_{ij}} \le \sum_{i = 1}^n \frac{D_i^2}{L_n} 
		+ 2\left(\sum_{i = 1}^n \frac{D_i^2}{L_n}\right)^2.
\]

First suppose that $1 < \gamma \le 2$. Then, by Corollary~\ref{cor:clt_edges} and the continuous
mapping theorem it follows that 
\[
	\frac{1}{L_n} \sum_{i,j = 1}^n \Expn{E_{ij}^c} \le \frac{1}{L_n^2} \sum_{i = 1}^n D_i^2 
		+ 2\frac{1}{L_n^3} \left(\sum_{i = 1}^n D_i^2\right)^2 = \bigOp{n^{\frac{4}{\gamma} - 3}}.
\] 

Now suppose that $\gamma > 2$. Then $D_i^2$ has finite mean, say $\nu$, and therefore, by Theorem
\ref{thm:stable_clt},
\[
	\frac{1}{L_n^2} \sum_{i = 1}^n D_i^2 \le \frac{1}{L_n^2}\left|\sum_{i = 1}^n D_i^2 - n\nu\right|
	+ \frac{n\nu}{L_n^2} = \bigOp{n^{\frac{2}{\gamma} - 2} + n^{-1}} = \bigOp{n^{-1}},
\]
where the last step follows since $2/\gamma - 2 < -1$ when $\gamma > 2$. Since this is the main
term the result follows.
\qed 
\end{proof}

Lemma~\ref{lem:first_upper_bound} provides the last two upper bounds from Theorem
\ref{thm:main_result}. However, as we mentioned before, the bound $\bigOp{n^{4/\gamma - 3}}$ is not
tight over the whole range $1< \gamma\le 2$ since for $1 < \gamma < 4/3$ we have 
$4/\gamma - 3 > 0$, and hence the upper bound diverges as $n \to \infty$ which is in disagreement 
with~\eqref{eq:erased_edges_convergence}. Therefore, there must exist another upper bound on the 
average erased number of edges, which goes to zero as $n \to \infty$ for all $\gamma > 1$. 
This new bound does not follow readily from the literature. Below we establish such upper bound and
explain the essential new ingredients needed for its proof.

\subsection{The upper bound $\bigOp{n^{\frac{1}{\gamma} - 1}}$}

We first observe that the number of erased edges between nodes $i$ and $j$ equals the total number 
of edges between the nodes minus one, if there is more than one edge. This gives,
\begin{align*}
	\frac{1}{L_n} \sum_{i,j = 1}^n \Expn{E^c_{ij}} 
	&= \frac{1}{L_n} \sum_{i,j = 1}^n \Expn{E_{ij} - \mathbbm{1}_{\{E_{ij} > 0\}}} \\
	&= \frac{1}{L_n} \sum_{i,j = 1}^n \Expn{E_{ij}} - \frac{1}{L_n} \sum_{i,j = 1}^n 
		\Expn{1 - \mathbbm{1}_{\{E_{ij} = 0\}}} \\
	&= 1 - \frac{n^2}{L_n} + \frac{1}{L_n} \sum_{i,j = 1}^n \Probn{E_{ij} = 0}. 
	\numberthis \label{eq:removed_edges_alternative}
\end{align*}

We can get an upper bound for $\Probn{E_{ij} = 0}$ from the analysis performed in 
\cite{Hofstad2005}, Section 4. Since the probability of no edges between $i$ and $j$ equals the
probability that none of the $D_i$ stubs connects to one of the $D_j$ stubs, it follows from 
equation (4.9) in~\cite{Hofstad2005} that
\begin{equation}
	\Probn{E_{ij} = 0} \le \prod_{k = 0}^{D_i - 1}\left(1 - \frac{D_j}{L_n - 2D_i - 1}\right) 
	+ \frac{D_i^2 D_j}{(L_n - 2D_i)^2}.
	\label{eq:no_edge_bound}
\end{equation} 
The product term in~\eqref{eq:no_edge_bound} can be upper bounded by $\exp\{-D_iD_j/E_n\}$. For the
second term we use that
\begin{align*}
	\frac{1}{L_n}\sum_{i,j = 1}^n \frac{D_i^2 D_j}{(L_n - 2D_i)^2} 
	&= \frac{1}{L_n^2}\sum_{i = 1}^n D_i^2\left(\frac{1}{1-\frac{2D_i}{L_n}}\right)^2
		\left(\frac{1}{L_n}\sum_{j = 1}^n D_j\right) \\
	&\le \frac{1}{L_n^2} \sum_{i = 1}^n D_i^2 = \bigOp{n^{\frac{2}{\gamma} - 2}}.
\end{align*}
Putting everything together we obtain
\begin{equation}
		\frac{1}{L_n}\sum_{i,j = 1}^n\Probn{E_{ij} = 0} \le \sum_{i,j = 1}^n 
		\exp\left\{-\frac{D_i^+ D_j^-}{L_n}\right\} + \bigOp{n^{\frac{2}{\gamma} - 2}}.
		\label{eq:no_edge_exponential_bound}
\end{equation}

We will use~\eqref{eq:no_edge_exponential_bound} to upper bound
\eqref{eq:removed_edges_alternative}. In order to obtain the desired result we will employ a 
Tauberian Theorem for regularly varying random variables, which we summarize first. We write 
$a\sim b$ to denote that $a/b$ goes to one in a corresponding limit.

\pagebreak

\begin{theorem}[Tauberian Theorem, \cite{Bingham1974}]\label{thm:tauberian_theorem}
	Let $X$ be a non-negative random variable with only finite mean. Then, for $1 < \gamma < 2$, the
	following are equivalent,
	\begin{enumerate}
		\item[i)] $\displaystyle \Prob{X > t} \sim L(t) t^{-\gamma} \quad \text{as } t \to \infty$,
		\item[ii)] $\displaystyle \frac{\Exp{X}}{t} - 1 + \exp\left\{-\frac{X}{t}\right\} \sim 
			L\left(\frac{1}{t}\right) t^{-\gamma} \quad \text{as } t \to \infty$.
	\end{enumerate}
\end{theorem}

We will first explain the idea behind the proof of the $\bigOp{n^{1/\gamma - 1}}$ bound. If we 
insert~\eqref{eq:no_edge_exponential_bound} into~\eqref{eq:removed_edges_alternative} we get
\begin{equation}
	\frac{1}{L_n}\sum_{i,j = 1}^n \Expn{E_{ij}^c} \le 1 - \frac{n^2}{L_n} 
		+ \frac{1}{L_n}\sum_{i,j = 1}^n\exp\left\{-\frac{D_i^+ D_j^-}{L_n}\right\} 
		+ \bigOp{n^{\frac{2}{\gamma} - 2}}.
	\label{eq:erased_edges_exponential_bound}
\end{equation}
The terms on the right side can be rewritten to obtian an expression that resembles an empirical
version of the left hand side of part ii) from Theorem~\ref{thm:tauberian_theorem}, with $t = L_n$
and $X = D_1 D_2$. Thus, the scaling of the average number of erased edges will be determined by 
the scaling that follows from the Tauberian Theorem and the Stable Law CLT. 

\begin{proposition}\label{prop:second_upper_bound}
	Let $G_n$ be a graph generated by the ECM with degree distribution
	\eqref{eq:degree_distribution} and $1 < \gamma < 2$. Then
	\begin{equation}
		\frac{1}{L_n} \sum_{i,j = 1}^n \Expn{E_{ij}^c} = \bigOp{n^{\frac{1}{\gamma} - 1}}.
	\label{eq:erased_edges_bound_2}
	\end{equation}
\end{proposition}

\begin{proof}
	We start with equation~\eqref{eq:erased_edges_exponential_bound}. Since the correction term here
	is of lower order, by extracting a factor $n^2/L_n$ from the other terms and using that $L_n = 
	\sum_{i = 1}^n D_i$, it suffices to show that
	\begin{equation} 
		\frac{n^2}{L_n}\left(\frac{1}{n^2}\sum_{i,j = 1}^n \frac{D_i D_j}{L_n} - 1
		+ \frac{1}{n^2} \sum_{i,j = 1}^n \exp\left\{-\frac{D_i D_j}{L_n}\right\}\right) 
		= \bigOp{n^{\frac{1}{\gamma} - 1}}.
		\label{eq:removed_edges_main_term}
	\end{equation}
	We first consider the term inside the brackets in the left hand side of
	\eqref{eq:removed_edges_main_term}.
	\begin{align}
		&\left|\frac{1}{n^2} \sum_{i,j = 1}^n \frac{D_i D_j}{L_n} - 1 + \frac{1}
			{n^2} \sum_{i,j = 1}^n \exp\left\{-\frac{D_i D_j}{L_n}\right\}\right| \notag \\
		&\le \frac{1}{n^2}\left|\frac{1}{L_n} - \frac{1}{\mu n}\right|\sum_{i,j = 1}^n D_i D_j 
			\label{eq:degree_estimation_linear}\\
		&\hspace{10pt}+ \frac{1}{n^2} \sum_{i,j = 1}^n\left|\exp\left\{-\frac{D_i D_j}{L_n}
			\right\} - \exp\left\{-\frac{D_i D_j}{\mu n}\right\}\right| 
			\label{eq:degree_estimation_exponential}\\
		&\hspace{10pt}+ \left|\frac{1}{n^2} \sum_{i,j = 1}^n \left(\frac{D_i D_j}{\mu n} - 1 + 
			\exp\left\{-\frac{D_i D_j}{\mu n}\right\}\right)\right| \label{eq:tauberian_estimation}
	\end{align}
	Since
	\[
		\frac{1}{n^2} \sum_{i,j = 1}^n\left|\exp\left\{-\frac{D_i D_j}{L_n}
			\right\} - \exp\left\{-\frac{D_i D_j}{\mu n}\right\}\right|
		\le \frac{1}{n^2}\left|\frac{1}{L_n} - \frac{1}{\mu n}\right| \sum_{i,j = 1}^n D_i D_j,
	\]
	it follows from Corollary~\ref{cor:clt_edges} that both~\eqref{eq:degree_estimation_linear} and
	\eqref{eq:degree_estimation_exponential} are $\bigOp{n^{\frac{1}{\gamma} - 2}}$. Next, observe 
	that the function $e^{-x} - 1 + x$ is positive which implies, by Markov's inequality, that 
	\eqref{eq:tauberian_estimation} scales as its average
	\begin{equation}
		\frac{\Exp{D_1 D_2}}{\mu n} - 1 + \Exp{\exp\left\{-\frac{D_1 D_2}{\mu n}\right\}}.
		\label{eq:tauberian_term}
	\end{equation}
	where $D_1$ and $D_2$ are two independent random variables with distribution
	\eqref{eq:degree_distribution} and $1 < \gamma < 2$, so that the product $D_1 D_2$ again has 
	distribution~\eqref{eq:degree_distribution} with the same exponent, see for instance the 
	Corollary after Theorem 3 in~\cite{Embrechts1980}. Now we use Theorem~\ref{thm:tauberian_theorem}
	to find that~\eqref{eq:tauberian_term}, and hence~\eqref{eq:tauberian_estimation} are 
	$\bigOp{n^{-\gamma}}$. Finally, the term outside of the brackets in
	\eqref{eq:removed_edges_main_term} is $\bigOp{n}$ and since $1 - \gamma < \frac{1}{\gamma} - 1$
	for all $\gamma > 1$, the result follows.
	\qed
\end{proof}

\section{Discussion}

The configuration model is one of the most important random graph models developed so far for 
constructing test graphs, used in the study of structural properties of, and processes
on, real-world networks. The model is of course most true to reality when it produces a simple 
graph. Because this will happen with vanishing probability for most networks, since these have 
infinite degree variance, the ECM can be seen as the primary model for a neutrally wired 
simple graph with scale-free degrees. The fact that the fraction of erased edges is vanishing, 
suffices for obtaining asymptotic structural properties and asymptotic behavior of network 
processes in the ECM. However, real-world networks are finite, albeit very large. Therefore, it is
important to understand and quantify how the properties and processes in a finite network are 
affected by the fact that the graph is simple. 

This paper presents the first step in this direction by providing probabilistic upper bounds for 
the number of the erased edges in the undirected ECM. This second order analysis shows that the
average number of erased edges by the ECM decays as $n^{-1}$ when the  variance of the degrees is 
finite. Since the ECM is computationally less expensive then the RCM and other sequential 
algorithms, this is a strong argument for using the ECM as a standard model for generating test 
graphs with given degree distribution. Especially since, in contrast to Markov-Chain Monte Carlo 
methods using edge swap mechanics, it is theoretically well analyzed. We also uncover a new 
transition in the scaling of the average number of erased edges for regularly varying degree 
distributions with only finite mean, in terms of the exponent of the degree distribution.

Based on the empirical results found by us in~\cite{Hoorn2015}, we conjecture that the bounds 
we obtained are tight, up to some slowly varying functions. Therefore, as a next step one could try
to prove Central Limit Theorems for the number of erased edges, using the bounds from Theorem
\ref{thm:main_result} as the correct scaling factors. These tools would make it possible to perform 
statistical analysis of properties on networks, using the ECM as a model for generating test 
graphs.

\bibliographystyle{splncs03}
\bibliography{bib/erasededgescm}

\end{document}